\documentclass{amsart}

\usepackage{amssymb,amsmath,url}
\usepackage{graphicx}
\usepackage[all]{xy}

\usepackage{pinlabel}   

\title{The Morse--Bott--Kirwan condition is local}

\author{Tara Holm}
\address{Dept.\ of Mathematics, Cornell University, Ithaca, NY 14853 USA}
\email{tsh@math.cornell.edu}

\author{Yael Karshon}
\address{Dept.\ of Mathematics, University of Toronto,
40 St.\ George Street, Toronto Ontario M5S 2E4, Canada}
\email{karshon@math.toronto.edu}

\date{\today}


\newif\ifdebug                                                      %
\debugfalse

\newcommand{\printname}[1]
   {\smash{\makebox[0pt]{\hspace{-1.0in}\raisebox{8pt}{\tiny #1}}}}
\newcommand{\labell}[1] {\ifdebug {\label{#1}\printname{#1}}
                        \else    {\label{#1}} \fi}

\def \acts   {\  \rotatebox[origin=c]{-90}{$\circlearrowright$}\  }

\DeclareMathOperator \supp {supp} 
\DeclareMathOperator \Hess {Hess} 
\DeclareMathOperator \Crit {Crit} 
\DeclareMathOperator \closure {closure}
\DeclareMathOperator \image {image}

\def \R {{\mathbb R}}
\def \F {{\mathbb F}}
\def \Z {{\mathbb Z}}
\def \C {{\mathbb C}}
\def \Q {{\mathbb Q}}
\def \calO {{\mathcal O}}
\def \calW {{\mathcal W}}
\def \tN {{\widetilde{N}}}

\def \inv {^{-1}}
\def \ssminus {\smallsetminus}

\def \del {\partial}

\numberwithin{equation}{section}

\newtheorem {Theorem}                   {Theorem} [section]
\newtheorem*{Theorem*}                  {Theorem}

\newtheorem {Lemma}[Theorem]           {Lemma}

\newtheorem* {Corollary*}                {Corollary}
\newtheorem {Proposition} [Theorem]    {Proposition}
\newtheorem* {Proposition*} {Proposition}

\newtheorem* {Lemma*}                    {Lemma}
\theoremstyle{definition}
\newtheorem{Definition}[Theorem]{Definition}
\newtheorem*{Definition*}{Definition}
\newtheorem{Remark}[Theorem]{Remark}
\newtheorem*{Remark*}{Remark}
\newtheorem{Example}[Theorem]{Example}

\def \t {\mathfrak t}
\def \h {\mathfrak h}
\def \g {\mathfrak g}
\DeclareMathOperator \Ad {Ad}

\begin{document}

\begin{abstract}
Kirwan identified a condition on a smooth function 
under which the usual techniques of Morse--Bott theory
can be applied to this function.
We prove that if a function satisfies this condition locally 
then it also satisfies the condition globally.
As an application, we use the local normal form theorem
to recover Kirwan's result
that the norm-square of a momentum map satisfies Kirwan's condition.
\end{abstract}

\maketitle

\section{Introduction}
\labell{sec:intro}

For a Hamiltonian action of a compact Lie group on a compact symplectic
manifold, a fundamental work of Frances Kirwan \cite{kirwan:book}
makes it possible to apply Morse theoretic techniques
to the norm-square of the momentum map.
The norm-square is not a Morse--Bott function;
components of its critical set might not even be smooth submanifolds.
Kirwan identifies a condition on a real valued function,
being \emph{minimally degenerate},
which is more general than the Morse--Bott condition,
and which nevertheless allows one to apply the machinery of Morse theory.
Nowadays, this condition
is sometimes called ``Morse--Bott in the sense of Kirwan.''

Applying Morse--theoretic arguments to the norm-square of a momentum map
is the main ingredient in the proof of Kirwan surjectivity, a pivotal result
in equivariant symplectic geometry and geometric invariant theory, 
as well as in the study
of the global topology of Hamiltonian
compact group actions.  These techniques played a central role 
in the  mathematical confirmation of physicists' predictions
for the structure of the cohomology ring of moduli spaces
of holomorphic vector bundles over Riemann surfaces \cite{JK}.

Kirwan's definition of a minimally degenerate function is not local:
it requires the set of critical points to be a disjoint union 
of closed subsets, each of which has a neighbourhood that satisfies
a certain condition.  If the set of critical points is not discrete,
then, in contrast to the Morse--Bott condition,
{\em a priori} it is not clear if one can tell that a function
is ``minimally degenerate'' by examining small neighbourhoods
of individual critical points.
This aspect of the definition
makes it difficult to check whether a function satisfies this condition.

The main result of this paper is that Kirwan's ``minimally degenerate''
condition actually \emph{is} a local condition:
if a function is minimally degenerate near each critical point,
then it is minimally degenerate (Theorem~\ref{main}).
As a corollary, we obtain a Morse--Lemma--type local characterization
of minimally degenerate functions (Theorem~\ref{local-coords}).

Section \ref{sec:statement} contains basic definitions and the statements of 
our main results, Theorems~\ref{main} and~\ref{local-coords}.
The main steps of the proof are formulated as Propositions~\ref{step1},
\ref{step2}, and \ref{step3}.  
Section \ref{sec:sketch} contains the statements of these propositions,
followed by proofs of the main theorems.
Sections~\ref{sec:proof-infinitesimal}, \ref{sec:compatibly fibred nd},
and~\ref{sec:proof-minimal} are devoted to the proofs of these propositions. 
In Section~\ref{sec:norm-square} we then use Theorem \ref{main}
and the local normal form theorem to re-prove Kirwan's result
that, for a Hamiltonian action of a compact Lie group,
the norm-square of a momentum map is minimally degenerate.
Finally, in Section~\ref{sec:Morse}, we mention some consequences of this fact.

\medskip

\noindent {\bf Acknowledgements.}  
The first author has been partially supported by Simons Foundation
Grant $\#208975$
and National Science Foundation Grant DMS--$1206466$.
The second author is partially supported
by the Natural Sciences and Engineering Research Council of Canada.
We would like to thank Gwyneth Whieldon for helping us resolve
a stubborn \LaTeX\ challenge.

\section{Statement of the main result}
\labell{sec:statement}

We begin by recalling Kirwan's definitions 
of a minimizing manifold
and of a minimally degenerate function.

Throughout this paper, we do not assume that the dimension
of a manifold is constant.  That is, by ``manifold'', we allow
a disjoint union of manifolds of different dimensions.
In particular, in Definition~\ref{def:minimizing}, 
if $C$ is not connected, we allow $N$ to have connected components
of different dimensions.  Similarly, we allow a vector bundle
over a manifold $W$ to have different ranks over different components of~$W$. 

\begin{Definition} \labell{def:minimizing}
Let $W$ be a smooth manifold and $f \colon\thinspace W \to \R$ 
a smooth function.
Let $C$ be a closed subset of $W$ on which $f$ is constant
and such that every point of $C$ is a critical point of $f$.
A \textbf{minimizing manifold} for $f$ along $C$ is a submanifold $N$ of $W$
that contains $C$ and that has the following two properties.
\begin{enumerate}
\item
For each point $x$ of $C$, the tangent space $T_xN$ is maximal
among subspaces of $T_xW$ on which the Hessian of $f$ is positive semidefinite.
\item
The restriction $f|_N$ of $f$ to $N$
attains its minimum exactly at the points of~$C$.
\end{enumerate}
\end{Definition}

\begin{Remark} \labell{one prime}
In Definition~\ref{def:minimizing}, Condition (1) can be replaced 
by the following condition:
\begin{enumerate}
\item[(1$'$)]
For each point $x$ of $C$, there exists a splitting
$T_xW = T_xN \oplus E_x$
such that $E_x$
is maximal among subspaces of $T_xW$ on which the Hessian of $f$
is negative definite.
\end{enumerate}
\end{Remark}

\begin{Definition} \labell{def:min-deg}
Let $W$ be a smooth manifold and $f \colon\thinspace W \to \R$ 
a smooth function.
We say $f$ is \textbf{minimally degenerate} if 
its critical set is a locally finite disjoint union of closed subsets 
along which there exist minimizing manifolds for $f$.
\end{Definition}

\begin{Remark}
Every Morse--Bott function is minimally degenerate.
Kirwan showed that one can apply the usual techniques of 
Morse--Bott theory to a minimally degenerate function 
even if the function is not a Morse--Bott function.
\end{Remark}

\begin{Remark} \labell{differences}
Definition \ref{def:min-deg} was proposed by Kirwan \cite[page~6]{kirwan:book},
with two differences.
First, Kirwan assumes that the set of critical points is a 
\emph{finite} disjoint union of closed subsets 
along which there exist minimizing manifolds for $f$.
Second, she requires that the minimizing manifolds be co-orientable.
As for the first difference, the weaker assumption ``locally finite''
is sufficient for our purposes;
it is equivalent to Kirwan's assumption ``finite'' when the manifold
is compact, as is the case for complex projective manifolds,
which is the case in which Kirwan was the most interested.
As for the second difference,
co-orientability is required for many consequences
(see Remark~\ref{rk:coorientable}) and is guaranteed in our 
main application in symplectic geometry;
it is not essential, though, when we study
the differential topological properties of minimally degenerate functions.
\end{Remark}

\begin{Remark} \labell{rk:coorientable}
The usual definition of a Morse--Bott function does not require the negative
normal bundle to a critical set to be orientable.  
However, to obtain Morse--Bott inequalities for ranks of cohomology groups 
with coefficients in a field $\F$, the negative normal bundles 
must be $\F$--{\bf orientable}.  
A minimally degenerate function
for which the normal bundles to the minimizing manifolds are $\F$--orientable
similarly
gives rise to Morse--Bott inequalities for cohomology with $\F$ coefficients.

Note that every vector bundle is $\Z/2\Z$--orientable, 
and that orientability is equivalent to $\Q$--orientability.
For further details, see \cite[\S 2.6]{Nicolaescu}.
\end{Remark}

\begin{Remark} \labell{restrict}
Let $M$ be a smooth manifold and $f \colon\thinspace M \to \R$
a smooth function. Suppose that $f$ is minimally degenerate.
Then, for every open subset $U$ of $M$,
the restriction $f|_U$ is minimally degenerate.
\end{Remark}

Our main result is that the existence of minimizing manifolds 
can be checked locally.  
Co-orientability of the minimizing manifolds 
is inherently a global property;
in the presence of an almost complex structure,
we give a condition that guarantees co-orientability 
and that can be checked locally.

\begin{Theorem} \labell{main}
Let $M$ be a smooth manifold and $f \colon\thinspace M \to \R$ 
a smooth function.
Suppose that every point in $M$ has an open neighbourhood $U$
such that $f|_U$ is minimally degenerate.  Then $f$ is minimally degenerate.

Suppose in addition that there exists an almost complex structure $J$ on $M$
such that at every critical point of $f$ the Hessian of $f$ is $J$--invariant.
Then $f$ is minimally degenerate with co-orientable minimizing manifolds.  
\end{Theorem}

Tolman and Weitsman note in \cite[p.\ 759]{TW} 
that minimally degenerate functions
``morally ... look like the product 
of a minimum and a non-degenerate Morse--Bott function''.
Our locality result, Theorem \ref{main},
allows us to make this description rigorous:

\begin{Theorem} \labell{local-coords}
Let $M$ be a smooth $n$--dimensional manifold and $f \colon\thinspace M \to \R$ 
a smooth function.  Then the following conditions are equivalent.
\begin{enumerate}
\item[(a)]
For every critical point $c$, 
there exist coordinates $x_1,\dots,x_k,y_{k+1},\dots,y_n$ centered at $c$ 
such that in a neighbourhood of $c$
$$
f = f(\mathbf{x},\mathbf{y}) = g(\mathbf{y}) -\sum_{j=1}^k x_j^2\ ,
$$
where $g$ is a smooth function 
with minimal value $g(0)$ and no other critical values.

\item[(b)]
For every critical point $c$, 
there exist coordinates $x_1,\dots,x_k,y_{k+1},\dots,y_n$ centered at $c$ 
such that in a neighbourhood of $c$
$$
f = f(\mathbf{x},\mathbf{y}) = g(\mathbf{y}) + h(\mathbf{x})\ ,
$$
where $g$ is a smooth function 
with minimal value $g(0)$ and no other critical values,
and where $h$ is a Morse--Bott function.

\item[(c)]
\ $f$ is minimally degenerate.
\end{enumerate}
\end{Theorem}

\section{Outline of the proofs}
\labell{sec:sketch}

In this section we outline the proofs 
of Theorems~\ref{main} and~\ref{local-coords}.
We begin with a technical lemma:

\begin{Lemma} \labell{locally finite}
Let $M$ be a smooth manifold and $f \colon\thinspace M \to \R$ 
a smooth function.
Suppose that every point in $M$ has an open neighbourhood $U$
such that $f|_U$ is minimally degenerate.  
Then the critical set of $f$ is a locally finite disjoint union
of closed subsets on which $f$ is constant.
\end{Lemma}

\begin{proof}
The definition of ``minimally degenerate''
implies that each point has a neighbourhood 
in which the function takes only finitely many critical values.
The conclusion of the lemma then holds
when we take the closed subsets to be the intersections 
$ f^{-1}(c_i) \cap \Crit f $
where $c_i$ are the critical values of $f$.
\end{proof}

{
To prove Theorem~\ref{main}, we fix a manifold $M$
and a smooth function $f \colon\thinspace M \to \R$ 
that is locally minimally degenerate.
Lemma~\ref{locally finite} gives a decomposition of the critical set
of $f$ into a locally finite disjoint union of closed subsets
on which $f$ is constant.  
We may now focus on one such a subset; call it $C$.
Thus, $C$ is closed in $M$ and has a neighbourhood $W$ whose intersection
with the critical set of $f$ is exactly $C$,  and $f$ is constant on $C$.
We need to prove that there exists a minimizing manifold for $f$ along $C$.
We will do this in three steps, 
formulated below as Propositions \ref{step1}, \ref{step2}, and \ref{step3}.

For these propositions, we now fix a manifold $W$ 
and a smooth function 
$$f \colon\thinspace W \to \R,$$
and we assume that $f$ is constant 
on the set $C := \Crit(f)$ of its critical points.

An \textbf{infinitesimally minimizing} manifold for $f$ along $C$ 
is a submanifold $N$ of~$W$ that contains $C$ and that satisfies
Property (1) of Definition~\ref{def:minimizing}:
\begin{quotation}
For each point $x$ of $C$, the tangent space $T_xN$ is maximal
among subspaces of $T_xW$ on which the Hessian of $f$ is positive semidefinite.
\end{quotation}

\begin{Proposition}[Infinitesimal minimal degeneracy] \labell{step1}
Suppose that every point $x$ of $C$ has an open neighbourhood $U_x$
such that there exists an infinitesimally minimizing manifold 
for $f$ along $C \cap U_x$
Then there exists an infinitesimally minimizing manifold $N$ for $f$ along $C$.

Suppose in addition that there exists an almost complex structure $J$ on $W$
such that at every point of $C$ the Hessian of $f$ is $J$--invariant. 
Then there exists such an~$N$ that is co-orientable.
\end{Proposition}

\begin{Definition} \labell{compatibly fibred}
A \textbf{compatibly fibred neighbourhood} of $C$
is a neighbourhood $U$ of $C$ together with a submersion 
$\pi \colon\thinspace U \to \calO$
such that at every point $x$ in $C$ 
the vertical tangent space $\ker d\pi|_x$
is maximal among subspaces of $T_xW$ on which the Hessian of $f$
is negative definite.
The compatibly fibred neighbourhood is \textbf{fibrewise orientable} if 
its vertical tangent bundle $\ker d\pi \to U$ is an orientable vector bundle. 
\end{Definition}

\begin{Proposition}[Compatibly fibrated neighbourhood] 
\labell{step2}
Suppose that there exists an 
infinitesimally minimizing manifold $N$ for $f$ along $C$.
Then $C$ has a compatibly fibred neighbourhood.

Suppose in addition that $N$ is co-orientable.
Then $C$ has a compatibly fibred neighbourhood that is fibrewise orientable.
\end{Proposition}

\begin{Proposition}[Minimal degeneracy] \labell{step3}
Let $\tN$ denote the set of fibrewise critical points
of a compatibly fibred neighbourhood 
$\pi \colon\thinspace U \to \calO$ of $C$.
Then, after possibly intersecting with a smaller neighbourhood of $C$,
the following is true.
\begin{enumerate}
\item[(a)]
$\tN$ is an 
infinitesimally minimizing manifold for $f$ along $C$. 

If in addition the compatibly fibred neighbourhood of $C$
is fibrewise orientable, then $\tN$ is co-orientable.

\item[(b)]
Suppose that every point $x$ of $C$ has an open neighbourhood $U_x$
such that there exists a 
minimizing manifold for $f$ along $C \cap U_x$.
Then $\tN$ is a minimizing manifold for $f$ along $C$.
\end{enumerate}
\end{Proposition}

We prove Proposition~\ref{step1} in Section~\ref{sec:proof-infinitesimal};
the main challenge in the proof is to ``patch together" minimizing manifolds
for neighbourhoods of points in $C$.
Proposition~\ref{step2} is a consequence of the tubular neighbourhood
theorem; we prove it in Section~\ref{sec:compatibly fibred nd}.
We prove Proposition~\ref{step3} in Section~\ref{sec:proof-minimal};
Part (a) is a consequence of the implicit function theorem;
Part (b) requires additional local arguments.

}

We conclude this section with proofs of Theorems~\ref{main}
and~\ref{local-coords}.

\begin{proof}[Proof of Theorem~\ref{main}]
Let $M$ be a manifold and $f \colon\thinspace M \to \R$ a smooth function.
Suppose that every point of $M$ has an open neighbourhood $U$
such that $f|_U$ is minimally degenerate.

By Lemma~\ref{locally finite},
the critical set of $f$ is a locally finite disjoint union
of closed subsets on which $f$ is constant.
Let $C$ be one of these subsets.
We may restrict our attention to an open neighbourhood $W$ of $C$
whose intersection with the set $\Crit(f)$ of critical points of $f$
is equal to $C$.

By Proposition~\ref{step1}, there exists an infinitesimally minimizing 
manifold $N$ for $f$ along $C$. 
By Proposition~\ref{step2},
$C$ has a compatibly fibred neighbourhood $\pi \colon\thinspace U \to \calO$.
By Proposition~\ref{step3},
after possibly restricting to a smaller neighbourhood of $C$,
the set $\tN$ of fibrewise critical points of $\pi$
is a minimizing manifold for $f$ along $C$, as required.

Now suppose, in addition, that there exists an almost complex structure $J$
on $M$ such that at every critical point of $f$ the Hessian of $f$ 
is $J$-invariant.
By Proposition~\ref{step1}, 
there exists an infinitesimally minimizing manifold $N$ for $f$ along $C$ 
that is coorientable.   By Proposition~\ref{step2},
$C$ has a compatibly fibred neighbourhood $\pi \colon\thinspace U \to \calO$
that is fibrewise orientable.
By Proposition~\ref{step3}, after possibly restricting 
to a smaller neighbourhood of $C$, the set $\tilde{N}$
of fibrewise critical points of $\pi$
is a minimizing submanifold for $f$ along $C$ and is coorientable, as required.
\end{proof}

\begin{proof}[Proof of Theorem~\ref{local-coords}]

Clearly, (a) implies (b): 
take $h(\mathbf{x}) = - \sum_{j=1}^k x_j^2$.

Suppose that (b) holds.   
Let $c$ be a critical point and $(\mathbf{x},\mathbf{y})$ coordinates 
as in~(b). 
By the Morse--Bott Lemma, 
after a further change of coordinates in $\R^k$,
we can bring $h$ to the form
$h(x) = h(0) + x_1^2 + \ldots + x_{\ell}^2 
 - x_{\ell+1}^2 - \ldots - x_m^2$ near $x=0$,
with $0 \leq \ell \leq m \leq k$.
Then, near $c$, the set of critical points is 
$$ \{ (x,y) \ | \ x_1 = \ldots = x_m = 0 \text{ and }  g(y)=g(0) \},$$
and $\{ x_{\ell+1} = \ldots = x_m = 0 \}$ 
is a minimizing submanifold for $f$ along this set.
Because $c$ was arbitrary, $f$ is locally minimally degenerate.
Theorem~\ref{main} guarantees that $f$ is (globally) minimally degenerate.  
That is, (c) holds.

To show that (c) implies (a), we need
to express a minimally degenerate function in local coordinates.
For this
we use a parametrized version of the Morse--Bott Lemma.
Complete details of the proofs of the Morse Lemma and the Morse--Bott Lemma 
are spelled out by Banyaga and Hurtubise in \cite{BH}.  
The parametrized version that we use is described by H\"ormander
in \cite[Lemma C.6.1]{hormander}).

Namely, suppose that $f$ is a minimally degenerate function, 
and let $c$ be a critical point.  
Choose coordinates
$x_1',\dots,x_k',y_{k+1},\dots,y_n$ on a neighbourhood~$U$ of~$c$,
centred at~$c$, such that
$\{ \mathbf{x'}=0 \}$ defines a minimizing submanifold
for $f$ along $U \cap \text{Crit} f$,
and such that the Hessian of $f$ is negative definite
on the subspace of $T_cM$ that is represented by $\R^k \times \{ 0 \}$.
Applying the parametrized version of the Morse lemma,
we find coordinates $(\mathbf{x},\mathbf{y})$ centred at $c$
in which $f$ has the desired form
$$
f = f(\mathbf{x},\mathbf{y}) = g(\mathbf{y}) -\sum_{j=1}^k x_j^2\ ,
$$
where $g$ is a smooth function.  Minimal degeneracy 
guarantees 
that $g$ must attain its minimum at $\mathbf{y}=0$
and, after possibly shrinking the neighbourhood of $c$,  
that $g$ has no critical values except its minimal value.
\end{proof}

\section{Existence of infinitesimally minimizing submanifolds}
\labell{sec:proof-infinitesimal}

The purpose of this section is to prove Proposition~\ref{step1}.

Let $W$ be a smooth manifold and $f \colon\thinspace W \to \R$ 
a smooth function. 
Assume that $f$ is constant on the set $C := \Crit(f)$ of its critical points.
For every point $x$ of~$C$, let $U_x$ be an open neighbourhood of $x$,
and let $N_x \subset U_x$ be an infinitesimally minimizing manifold
for $f$ along $C \cap U_x$.
To prove the proposition, 
we need to find an infinitesimally minimizing submanifold $N$ for $f$ along $C$,
and, in the presence of an appropriate almost complex structure,
to show that $N$ is co-orientable.

We would like to obtain such an $N$ by
``patching together" the submanifolds $N_x$.
A priori it is not clear how to ``patch together'' submanifolds.
We can ``patch together'' functions, by means of a partition of unity,
so our first attempt is to express each $N_x$ as the regular zero set 
of a function 
$h_x \colon\thinspace U_x \to \R^k$ and to take $h := \sum \rho_i h_{x_i}$,
where $\{ \rho_i \}$
is a partition of unity on the union of the sets $U_x$
with $\supp \rho_i \subset U_{x_i}$.
To guarantee that zero remains a regular value of $h$ near $C$,
we require that the differentials of $h_{x_i}$ and $h_{x_j}$
coincide at the points of $C \cap U_{x_i} \cap U_{x_j}$.
If this can be arranged then 
$$
\big\{ h=0 \big\} \cap \big\{ \text{an appropriate neighbourhood of $C$} \big\}
$$
is a 
minimizing manifold for $f$ along $C$.

One problem with this approach is that
a regular level set of a function to $\R^k$
must have a trivial normal bundle.
So this method already cannot work in the case 
that $f$ is a Morse--Bott function
and the negative normal bundle of $C$ is nontrivial.
To fix this, instead of working with functions to $\R^k$,
we work with sections of a rank $k$ vector bundle~$E$.
We carry out this plan in the following three lemmas.

First, we find a sub-bundle of the tangent bundle on which the Hessian
is negative definite.

\begin{Lemma} \labell{extend Hessian}
There exists a neighbourhood $U$ of $C$,
and a sub-bundle $E$ of $TW|_U$,
such that at each point $x$ of $C$
the subspace $E_x$ of $T_xW$ is maximal
among subspaces 
on which the Hessian $\Hess f|_x$ is negative definite. 

Moreover, let $J$ be an almost complex structure on $W$
such that at every point of $C$ the Hessian $\Hess f|_x$
is $J$--invariant.  Then the bundle $E$ may be chosen to be complex, 
hence orientable.
\end{Lemma}

\begin{proof}
We will construct such a bundle $E$ as a sum of eigenbundles
of a fibrewise automorphism $A \colon\thinspace TW \to TW$, near $C$.

First, we will extend the Hessian of $f$, which is only defined at critical
points, to a symmetric $2$--tensor $B$ that is defined on all of $W$.
For this, let $\{ U_\alpha \}$
be domains of coordinate charts that cover $W$;
let $B_\alpha$ be the symmetric $2$--tensor on $U_\alpha$ 
that in the local coordinates on $U_\alpha$ 
is represented by the matrix of second partial derivatives of $f$;
and take $B = \sum \rho_\alpha B_\alpha$,
where $\{ \rho_\alpha \colon\thinspace W \to \R \}$
is a partition of unity with $\supp \rho_\alpha \subset U_\alpha$.

Next, choose a Riemannian metric $\left< \cdot , \cdot \right>$ on $W$,
and define $A \colon\thinspace TW \to TW$ by $B(u,v) = \left< u ,Av \right>$.
Because $B(\cdot,\cdot)$ is symmetric, $A$ is self adjoint
with respect to $\left< \cdot , \cdot \right>$, and so $A$ is diagonalizable. 
For each $x' \in W$, let 
$\lambda_{1,x'}, \ldots, \lambda_{n,x'}$
denote the eigenvalues of $A|_{x'}$, 
in (weakly) increasing order.
For each~$i$, the eigenvalue $\lambda_{i,x'}$
is continuous, but perhaps not smooth, as a function of $x'$.

Our assumptions imply that for every $x \in C$
there exists a neighbourhood $U_x$ and an integer $k_x$ 
(namely, the codimension of $N_x$)
such that, for every $x' \in U_x \cap C$, 
the automorphism $A$ of $T_{x'}W$ has exactly~$k_x$ negative eigenvalues.

Let $C_k$ be the subset of $C$ where $k_x = k$.  
It is closed in $W$.
There exists a neighbourhood $U_k$ of $C_k$ in $W$ 
such that for all $x' \in U_k$
the eigenvalue $\lambda_{k+1,x'}$ is strictly greater
than the eigenvalues $\lambda_{1,x'}, \ldots, \lambda_{k,x'}$.
Shrink the sets $U_k$ so that their closures become disjoint.
For each $x' \in U_k$, let $E_{x'}$ be the sum of the eigenspaces
of $A|_{x'} \colon\thinspace T_{x'}W \to T_{x'}W$
that correspond to the eigenvalues $\lambda_{1,x'}, \ldots, \lambda_{k,x'}$.
Let $U = \bigcup_k U_k$. Then $E$ is a smooth subbundle of $TW|_U$,
and at each point $x$ of $C$, \ $E_x$ is a maximal subspace of~$T_xW$
on which $\Hess f|_x$ is negative definite.

Finally, suppose that the Hessian is $J$--invariant.
By averaging, we can arrange the tensor $B$ 
and the Riemannian metric $\left< \cdot , \cdot \right>$ 
to be $J$--invariant as well. The automorphism $A \colon\thinspace TW \to TW$
is then complex linear,
so its eigenbundles are $J$--invariant, and $E$ is a complex,
and hence orientable, vector bundle.
\end{proof}

Let $h \colon\thinspace U \to E|_U$ be a smooth section of a vector bundle $E$,
and let $x \in U$ be a point where this section vanishes.
The \textbf{vertical differential} of $h$ at $x$,
$$ d_Vh \colon T_x W \to E_x, $$
is the composition of the differential
$dh|_x \colon\thinspace T_xW \to T_{(x,0)} E$ 
with the projection to the second factor in the decomposition
\begin{eqnarray*}
T_{(x,0)} E  & \cong & T_{(x,0)} \big(\text{the zero section of $E$}\big) 
\oplus T_{(x,0)} \big(\text{the fiber $E_x$ of $E$}\big)\\
& \cong & T_xW \oplus E_x .
\end{eqnarray*}

In the second lemma, we find an appropriate section of the bundle $E$
of Lemma~\ref{extend Hessian}
whose zero-set will give us the desired submanifold $N$.

\begin{Lemma} \labell{exists section}
Let $U$ be a neighbourhood of $C$ and $E$ a subbundle of $TW|_U$,
such that at each point $x$ of $C$
the subspace $E_x$ of $T_xW$ is maximal
among subspaces on which $\Hess f|_x$ is negative definite. 

Then, after possibly shrinking $U$ to a smaller neighbourhood of $C$, 
there exists a smooth section $h \colon\thinspace U \to E|_U$
that vanishes on $C$
and such that at each critical point $x \in C$
the restriction to $E_x$ of the vertical differential of $h$
is the identity map on~$E_x$, 
and $\Hess f|_x$ is positive semidefinite on $\ker(d_Vh|_x)$.
\end{Lemma}

\begin{proof}
Recall that, for each $x \in C$, 
 \ $U_x$ is a neighbourhood of $x$
and $N_x \subset U_x$ is an infinitesimally minimizing manifold 
for $f$ along $C \cap U_x$.

Let $x \in C$.   Let
$$ \varphi \colon\thinspace U'_x \to \R^n $$
be a coordinate chart on a neighbourhood $U'_x$ of $x$ in $U_x \cap U$
in which the submanifold $N_x$ is given by the equations
$\varphi_1 = \ldots \varphi_k = 0$.
Because $T_x N_x$ is complementary to~$E_x$,
after possibly shrinking $U'_x$,
the differentials of $\varphi_1,\ldots,\varphi_k$
give a trivialization of $E|_{U'_x}$:
$$ E|_{U'_x} \to \R^k .$$
Let 
$$ h_x \colon\thinspace U'_x \to E|_{U'_x} $$
be the section whose composition with the trivialization $ E|_{U'_x} \to \R^k $
is the map $(\varphi_1,\ldots,\varphi_k)$.
Then $h_x$ vanishes on $C \cap U'_x$, 
and, at each $x' \in C \cap U'_x$, the restriction to $E_{x'}$ 
of the vertical differential $d_Vh_x|_{x'}$ is the identity map on $E_{x'}$, 
and $\Hess f|_{x'}$ is positive semidefinite on $\ker(d_Vh_x|_{x'})$.

Let $U' = \bigcup_{x \in C} U'_x$.
Define a section $h \colon\thinspace U' \to E$ by
$$ h = \sum_\alpha \rho_\alpha h_{x_\alpha} $$
where $\{ \rho_\alpha \colon\thinspace U' \to \R \}$
is a partition of unity with $\supp \rho_\alpha \subset U'_{x_\alpha}$.
Then $h$ satisfies the required properties.

Indeed, let $x' \in C$.  
Then $h(x') = \sum_\alpha \rho_\alpha h_{x_\alpha} (x')$.
Since $h_{x_\alpha} (x') = 0$ for each $\alpha$,
we get that $h(x') = 0$.
Choosing a local trivialization of $E$ near $x'$,
and identifying sections of $E$ with $\R^k$ valued functions,
for every $v \in T_{x'}W$ we have
$$ dh|_{x'} (v) = \sum_\alpha 
 \rho_\alpha d h_{x_\alpha}|_{x'} (v)
 + h_{x_\alpha} (x') d \rho_\alpha|_{x'} (v).$$
Because $h_{x_\alpha}(x')=0$ for every $\alpha$, we get 
$d_Vh|_{x'}  = \sum_{\alpha}\rho_\alpha d_Vh_{x_\alpha}|_{x'}.$
The required properties of $d_Vh|_{x'}$ follow from the analogous properties 
of $d_Vh_{x_\alpha}|_{x'}$.  (This uses the following linear algebra fact. 
Let $V$ be a vector space, $B$ a symmetric bilinear form on $V$, 
and $E$ a subspace of $V$ on which $B$ is negative definite. 
For every $\alpha$, let $H_\alpha \colon V \to E$ be a projection map 
such that $B$ is positive semidefinite on $\ker H_\alpha$.
Let $H = \sum_\alpha \rho_\alpha H_\alpha$, 
where $\rho_\alpha$ are non-negative real numbers with sum $1$.
Then $H \colon V \to E$ is a projection map 
such that $B$ is positive semidefinite on $\ker H$.)
\end{proof}

Finally, we use the zero-set of the section found in Lemma~\ref{exists section}
to obtain an infinitesimally minimizing submanifold $N$ for $f$ along $C$.

\begin{Lemma} \labell{exists inf minimizing}
Let $U$ be a neighbourhood of $C$ and $E$ a subbundle of $TW|_U$,
such that at each point $x$ of $C$
the subspace $E_x$ of $T_xW$ is maximal
among subspaces on which $\Hess f|_x$ is negative definite.

Let $h \colon\thinspace U \to E|_U$ be a smooth section that vanishes on $C$
and such that at each critical point $x \in C$
the restriction to $E_x$ of the vertical differential of $h$
is the identity map on~$E_x$
and $\Hess f|_{x}$ is positive semidefinite on $\ker (d_Vh|_x)$.

Then, after possibly shrinking the neighbourhood $U$ of~$C$, 
the set $N = h\inv(0)$
is an infinitesimally minimizing submanifold for $f$ along $C$.

If in addition, $E$ is an orientable vector bundle, 
then $N$ is co-orientable.
\end{Lemma}

\begin{proof}
Fix a point $x \in C$,   
and trivialize $E$ in a neighbourhood $U_x$ of $x$.
In terms of this trivialization, $h$ becomes a map from $U_x$ to $\R^k$. 
The hypotheses guarantee that 
the differential of $h$ at $x$ is onto.
By the implicit function theorem, after possibly shrinking $U_x$, 
the intersection $U_x \cap h\inv(0)$ is a submanifold of $U_x$.
So we have shown that every point in $C$ has a neighbourhood $U_x$ 
whose intersection with $N = h\inv(0)$ is a submanifold of $U_x$. 
After possibly shrinking $U$,
we obtain that $N$ itself is a submanifold of~$U$.

Let $x\in C$.  By the implicit function theorem, $T_xN = \ker (d_Vh|_x)$.
Because $\mathrm{Hess} f|_{x}$ is positive semidefinite on this space and 
negative definite on $E_x$, and because these spaces are complementary,
$T_xN$ is maximal among subspaces of $T_xW$ on which 
$\mathrm{Hess} f|_x$ is positive semidefinite.

Finally, we note that if $E$ is an orientable vector bundle, 
then $N$ is co-orientable.
Indeed, we have constructed $N$ so that its normal bundle 
is the pullback of $E$,
and co-orientability is precisely orientability of the normal bundle.
\end{proof}

This completes the proof of Proposition~\ref{step1}.  

\section{Compatibly fibrating neighbourhood}
\labell{sec:compatibly fibred nd}

The purpose of this section is to prove Proposition~\ref{step2}.

Let $W$ be a smooth manifold and $f \colon\thinspace W \to \R$
a smooth function.  Assume that $f$ is constant on the set $C := \Crit(f)$
of its critical points.  
Let $N$ be an infinitesimally minimizing manifold for $f$ along $C$.
To prove the proposition, 
we need to find a compatibly fibred neighbourhood for $C$,
which is fibrewise orientable if $N$ is co-orientable.

Clearly, for every point $x$ of $C$,
there exists an open neighbourhood $U_x$ of $x$
and an infinitesimally minimizing manifold $N_x \subset U_x$
for $f$ along $C \cap U_x$: take $U_x = W$ and $N_x = N$.
By~Lemma~\ref{extend Hessian},
we conclude that there exists a neighbourhood $U$ of~$C$
and a sub-bundle $E$ of $TW|_U$
such that at each point $x$ of $C$ 
the subspace $E_x$ of $T_xW$ is maximal
among subspaces on which the Hessian of $f$ is negative definite.

The exponential map with respect to any Riemannian metric
identifies a disc sub-bundle of $E|_{N \cap U}$ 
with a (tubular) neighbourhood $U'$ of $N \cap U$ in $U$.
Let $\pi \colon U' \to N \cap U$ be the corresponding projection map.
For each $x \in C$, the tangent to the fibre of $\pi$ at $x$ is $E_x$,
which is maximal among subspaces of $T_xW$ on which the Hessian of $f$
is negative definite.
So $\pi \colon U' \to N \cap U$ is a compatibly fibred neighbourhood of $C$. 

Here, $N \cap U$ serves two roles: it is a submanifold of $W$ that contains~$C$,
and it is the base of the fibration $\pi$.  
To match the notation of Definition~\ref{compatibly fibred},
we use a different symbol for the base of the fibration.
Namely, we write $\calO$ instead of $N \cap U$ for the base of the fibration,
so that the tubular neighbourhood map 
becomes a map $\pi \colon\thinspace U' \to \calO$,
and this map is a compatibly fibred neighbourhood of $C$.

Finally, we consider the case when $N$ is co-orientable.  
An orientation on the fibres of the normal bundle of $N$ in $W$
gives an orientation on the fibres
of the compatibly fibred neighbourhood that we have constructed.
Thus, the compatibly fibred neighbourhood is fibrewise orientable, 
as desired.

This completes the proof of Proposition~\ref{step2}.

\section{Minimal degeneracy}
\labell{sec:proof-minimal}

The purpose of this section is to prove Proposition~\ref{step3}.

Let $W$ be a smooth manifold and $f \colon\thinspace W \to \R$
a smooth function.  Assume that $f$ is constant on the set $C:=\Crit(f)$
of its critical points.
Let $\pi \colon\thinspace U \to \calO$ be a compatibly fibred
neighbourhood of $C$.  Namely, $U$ is a neighbourhood of $C$ in $W$,
\ $\pi$ is a submersion, and at every point $x$ in $C$
the vertical tangent space at $x$ is maximal
among subspaces of $T_xW$ on which the Hessian of $f$ 
is negative definite.
Let $\tN$ denote the set of fibrewise critical points of $\pi$, that is, 
the points whose vertical tangent space is in the kernel of $df$.

\begin{Lemma} \labell{local properties}
For every point $x$ of $C$ there exists a neighbourhood $W_x$
such that the intersection $N_{W_x} := W_x \cap \tN$
has the following properties.
\begin{itemize}
\item $N_{W_x}$ is a manifold, containing $x$.
\item $T_x N_{W_x}$ is maximal among subspaces of $T_xW$
on which the Hessian of $f$ is positive semidefinite.
\end{itemize}
\end{Lemma}

\begin{proof}
Because every critical point is fibrewise critical, $x$ is in $\tN$.
Because $\pi$ is a submersion, without loss of generality
we may identify a neighbourhood of $x$ in $W$
with a neighbourhood of the origin in $\R^a \times \R^b$,
such that $x$ becomes the origin, and such that $\pi$ becomes
the projection map

$$ \pi (\xi_1 , \ldots , \xi_a , \eta_1,\ldots,\eta_b )
   = (\xi_1 , \ldots , \xi_a) .$$
The vertical differential of $f$ then becomes the function
$$ d_V f \colon\thinspace \R^a \times \R^b \to \R^b $$
that is given by
$$ \left( \frac{\del f}{\del \eta_1} , \ldots , 
          \frac{\del f}{\del \eta_b} \right) .$$
The set $\tN$ of fibrewise critical points is precisely the zero set of $d_V f$.
By the implicit function theorem, to show that $\tN$ is a manifold 
near $x$, it is enough to show that the differential 
of $d_Vf$ at the origin,
\begin{equation} \labell{differential of dVf}
 d (d_V f)|_0 \colon\thinspace \R^a \times \R^b \to \R^b,
\end{equation}
is onto.
In coordinates, the linear map~\eqref{differential of dVf}
is represented by the $(a+b) \times b$ matrix
$$
\left( 
\begin{array}{cccccc}
\frac{\del^2 f}{ \del \xi_1 \del \eta_1} \Big|_0 & \ldots &
\frac{\del^2 f}{ \del \xi_a \del \eta_1} \Big|_0 & 
\frac{\del^2 f}{ \del^2 \eta_1} \Big|_0 & \ldots &
\frac{\del^2 f}{ \del \eta_b \del \eta_1} \Big|_0 \\
\vdots & & \vdots & \vdots & & \vdots \\
\frac{\del^2 f}{ \del \xi_1 \del \eta_b} \Big|_0 & \ldots &
\frac{\del^2 f}{ \del \xi_a \del \eta_b} \Big|_0 & 
\frac{\del^2 f}{ \del \eta_1 \del \eta_b} \Big|_0 & \ldots &
\frac{\del^2 f}{ \del^2 \eta_b } \Big|_0 
\end{array}
\right) .
$$
The fact that the right $b \times b$ block of this matrix is negative definite,
hence non-degenerate, implies that~\eqref{differential of dVf} is onto,
as required.

This fact also implies that the kernel of~\eqref{differential of dVf} 
is a complementary subspace to $\{ 0 \} \times \R^b$
in $\R^a \times \R^b$,
thus, that $T_xN_{W_x}$ is a complementary subspace to $\ker d\pi|_x$
in $T_xW$.
After a further change of coordinates,
we can arrange that $T_x N_{W_x}$ is represented
by the subspace $\R^{a} \times \{ 0 \}$ of $\R^a \times \R^b$.

In these new coordinates, because $\frac{\del f}{\del \eta_j}$ vanishes
along $\tN$ and $\frac{\del}{\del \xi_i}$ is tangent to $\tN$ at~$x$,
we get that $\left.\frac{\del^2 f}{\del \xi_i \del \eta_j}\right|_0 = 0$
for all $1 \leq i \leq a$ and $1 \leq j \leq b$.
So, in these coordinates, the Hessian of $f$ at $x$
becomes the bilinear form that is represented by the block diagonal matrix
$$ \left(
\begin{array}{cc}
\left.\frac{\del^2 f}{\del\xi_i\del\xi_{i'}}\right|_0 & 0 \\
 0 & \left.\frac{\del^2 f}{\del \eta_j \del \eta_{j'}}\right|_0
\end{array}
\right).
$$
By assumption, $\{ 0 \} \times \R^b$ is maximal among subspaces of $\R^{a+b}$
on which this matrix is negative definite.
Because the matrix is block diagonal, it follows that $\R^a \times \{ 0 \} $
is maximal among subspaces of $\R^a \times \R^b$ on which
this matrix is positive semidefinite.  That is, $T_x N_{W_x}$
is maximal among subspaces of $T_xW$ on which the Hessian of $f$
is positive semidefinite, as required.
\end{proof}

Following the notation of Lemma~\ref{local properties},
we set $U' := \bigcup W_x$.  Then $U'$ is an open neighbourhood of $C$,
the map $\pi' := \pi|_{U'} \colon\thinspace U' \to \calO$ is a submersion,
and $\tN' := \tN \cap U'$
is the set of fibrewise critical points of~$\pi'$.

The sets $N_{W_x}$ form an open covering of $\tN'$.
By Lemma~\ref{local properties}, 
each of them is a manifold. We deduce that $\tN'$ is a manifold.
Moreover, because every critical point is fibrewise critical,
$N_{W_x}$ contains $W_x \cap C$; it follows that $\tN'$ contains $C$.
For each $x \in C$, by Lemma~\ref{local properties},
and since $T_x\tN' = T_xN_{W_x}$, we have that $T_x\tN'$ is maximal
among subspaces of $T_xW$ on which the Hessian of $f$ is positive
semidefinite.
Moreover, when the compatibly fibred neighbourhood is 
fibrewise orientable,
the vertical tangent bundle $\ker d\pi|_{\tN'}$ is by definition orientable.
But this bundle is isomorphic to 
the normal bundle to $\tN'$ in $W$, and
so $\tN'$ is co-orientable.
This completes the proof of part (a) of Proposition~\ref{step3}.

Since the vertical tangent bundle $\ker d\pi |_{\tN'}$
is complementary to $T\tN'$ in $TW|_{\tN'}$,
there exists an open neighbourhood $U_{\tN'}$ of $\tN'$ in $U'$
and a tubular neighbourhood map
$$
 \pi_{\tN'} \colon\thinspace U_{\tN'} \to \tN' 
$$
whose fibres are open subsets of the fibres of $\pi$.

More precisely, 
let $E$ be the pullback to $\tN'$ of the vertical tangent bundle,
that is, for each $x' \in \tN'$, the fibre of $E$ at $x'$
is $\ker d\pi|_{x'}$.
Choose a fibrewise Riemannian metric on $U'$.
Then the fibrewise exponential map gives a diffeomorphism
from a neighbourhood of the zero section in $E$
to a neighbourhood $U_{\tN'}$ of $\tN'$ in $W$
that carries the projection map to $\pi_{\tN'}$
when we identify $\tN'$ with the zero section.

With the identification of~$U_{\tN'}$ with an open subset of a vector bundle,
the fibrewise second derivative of $f$ becomes well defined on~$U_{\tN'}$. 
The set of points where this fibrewise second derivative of $f$ 
is negative definite
is an open neighbourbood of~$\tN'$ in~$U_{\tN'}$.
This, and the fact that the set of fibrewise critical points
is exactly~$\tN'$, together imply 
that, after possibly shrinking~$U_{\tN'}$ to a smaller neighbourhood of~$\tN'$,
the fibrewise maxima of $f$ are achieved exactly at the points of~$\tN'$.

Abusing notation by returning to our previous symbols,
we now assume that we have the following set--up:
\ $\tN$ is a submanifold of $W$ that contains $C$;
\ $U$ is an open neighbourhood of $\tN$;
\ $\pi \colon\thinspace U \to \tN$ is a tubular neighbourhood map; 
\ the set of fibrewise critical points of $\pi$ is exactly $\tN$; 
\ the fibrewise maxima of $f$ are achieved exactly at the points of $\tN$;
\ and, at each point $x$ of $C$, we have that $\ker d\pi|_x$
is maximal among subspaces of $T_xW$ on which the Hessian of $f$
is negative definite.

\medskip

Now we want to show that $f|_{\tN}$ attains its minimum value exactly 
at the points of~$C$, 
after possibly intersecting with a smaller neighbourhood of $C$.
For every point $x$ of $C$, let $U_x$ be an open neighbourhood
of $x$ and $Z_x \subset U_x$ a minimizing submanifold 
for the function $f|_{U_x} \colon\thinspace U_x \to \R$
along the closed subset~$C \cap U_x$.
Since $\ker d\pi|_x$ is maximal among subspaces of $T_xW$
on which the Hessian of $f$ is negative definite
and $T_x Z_x$ is maximal among subspaces of $T_xW$ 
on which the Hessian of $f$ is positive semidefinite,
$Z_x$ is transverse to the fibres of $\pi|_{U_x}$ at the point $x$.
Therefore, $\pi|_{Z_x} \colon\thinspace Z_x \to \tN$ is a submersion
at the point $x \in Z_x$, so $\pi(Z_x)$ contains a neighbourhood
of $x$ in $\tN$.
So $\bigcup_x \pi(Z_x)$ contains a neighbourhood of $C$ in $\tN$.
Let $\calO'$ be such a neighbourhood.
By the choice of $Z_x$, the restriction $f|_{Z_x}$ attains
its minimum exactly on $Z_x \cap C$.
Thus, for every $y \in Z_x \ssminus C$, we have $f(y) > f(x)$.
But we have also arranged that the fibrewise maxima of $f$
are attained exactly on $\tN$.
So, for every $y \in Z_x$, we have $f(y) \leq f(\pi(y))$.
We conclude that, for every $y \in Z_x$,
we have $f(\pi(y)) \geq f(x)$, and equality implies that $y \in C$,
which implies that $\pi(y) = y$ and hence that $\pi(y) \in C$.
It follows that $f|_{\pi(Z_x)}$ attains its minimum exactly on $Z_x \cap C$.
Hence, $f|_{\calO'}$ attains its minimum exactly on $C$.
So $\calO'$ is a minimizing manifold for $f$ along~$C$, as required.

This completes the proof of Proposition~\ref{step3}.

\section{Norm-square of the momentum map}
\labell{sec:norm-square}

Let $\Phi \colon\thinspace M \to \g^*$ be the momentum map
for the action of a compact Lie group $G$ on a symplectic manifold $(M,\omega)$.
Fix an $\Ad$--invariant inner product on~$\g$, and fix the induced
inner product on $\g^*$.   

\begin{Theorem} \labell{norm square}
The function $\| \Phi \|^2 \colon\thinspace M \to \R$ is minimally degenerate.
\end{Theorem}

Theorem~\ref{norm square} was proved by Kirwan in \cite{kirwan:book}.
We will present here a slightly different proof: 
Theorem~\ref{main} reduces the problem to a local result,
and we deduce this local result
from the local normal form theorem for Hamiltonian $G$ actions.

\begin{Remark} \ 
\begin{enumerate}
\item
We recall what it means for $\Phi \colon\thinspace M \to \g^*$ to be 
a momentum map.  First, 
for every $\xi$ in the Lie algebra $\g$ of $G$, 
denoting the corresponding vector field by $\xi_M$,
we have \textbf{Hamilton's equation}:
\begin{equation} \labell{hamilton}
  d \left< \Phi , \xi \right> = \iota_{\xi_M} \omega.
\end{equation}
Second, $\Phi$ intertwines the $G$ action on $M$
with the coadjoint $G$ action on the dual $\g^*$ of the Lie algebra $\g$.

The local normal form theorem gives an explicit formula 
for the $G$--action, the symplectic form $\omega$,
and the momentum map $\Phi$, on a neighbourhood of a $G$ orbit.

\item
Kirwan's motivation for generalizing Morse(--Bott) theory
was in fact to apply such a theory to the norm-square of the momentum map,
as suggested in the work of Atiyah and Bott \cite{AB}.
In Section \ref{sec:Morse} we recall some of the consequences
of this application.

\item
Suppose that $G$ is a torus, $M$ is compact, and $M$ has a $G$ invariant
K\"ahler metric.  The compactness of $M$ implies that the gradient
flows of the components of the momentum map $\Phi$ are defined
for all times, and the integrability of the complex structure
implies that these flows commute.
In this case, Kirwan proves
in \cite{kirwan:article} that $f \circ \Phi$ is minimally degenerate
for \emph{any} convex function $f$; in particular, $\| \Phi \|^2$
is minimally degenerate.  So in this case one does not need
the full power of Kirwan's analysis in \cite{kirwan:book}
nor our local analysis here.

\item Notice that a connected component of the critical set of $\| \Phi\|^2$
need not be a manifold. For example, the circle action  $S^1\acts \C^2$
with weights $1$ and $-1$ has momentum map
\begin{eqnarray*}
\Phi: \C^2 & \to & \R\\
(z,w) & \mapsto & \frac{|z|^2}{2}-\frac{|w|^2}{2}.
\end{eqnarray*}
Note that this is a Morse--Bott function, and it has a critical value at $0\in\R$. 
The critical set for $\Phi$ is $\{ 0\}\subset\C^2$, which is a submanifold.
The norm-square $\| \Phi\|^2$ also has a critical
value at $0$, and for the norm-square, the critical set is
$$
\left\{ (z,w)\in\C^2\ \Bigg| \  \frac{|z|^2}{2}-\frac{|w|^2}{2}=0\right\},
$$
which is not a manifold.
\end{enumerate}
\end{Remark}

\smallskip

We now recall some general criteria for identifying the critical set.

\begin{Lemma}[Kirwan {\cite[$\S3$]{kirwan:book}}] \labell{criterion}
Let $\beta = \Phi(p)$.  Let $T_\beta$ be the closure in $G$ 
of the one parameter subgroup that is generated by the element  
of $\g$ that corresponds to $\beta$ by the inner product.
Let $\h$ denote the Lie algebra of the stabilizer of $p$;
let $\h^*$ be its dual, embedded in $\g^*$ by the inner product.
The following conditions are equivalent.
\begin{enumerate}
\item[(i)]
$p \in \Crit \| \Phi\|^2$.
\item[(ii)]
$\beta \perp \image d\Phi|_p$.
\item[(iii)]
$\beta \in \h^*$.
\item[(iv)]
$p$ is fixed by $T_\beta$.
\end{enumerate}
\end{Lemma}

\begin{proof}
Since $\| \Phi \|^2 = \left< \Phi , \Phi \right>$,
we have $d \| \Phi \|^2 = 2 \left< d\Phi , \Phi \right>$.
So 
$$ d \| \Phi \|^2 |_p (v) = 2 \left< d\Phi|_p(v) , \Phi(p) \right>.$$
This vanishes for all $v \in T_pM$ exactly 
if every element in the image of $d\Phi|_p \colon\thinspace T_pM \to \g^*$
is perpendicular to $\Phi(p)$.
This shows that (i) is equivalent to~(ii).

The subset of $\g^*$ that is identified with $\h^*$ by the inner product
is exactly the 
orthocomplement of the annihilator $\h^0$ of $\h$ in $\g^*$.
But, since $\h$ is the Lie algebra of the stabilizer of $p$,
the image of $d\Phi|_p \colon\thinspace T_pM \to \g^*$ is exactly equal
to $\h^0$; this is a consequence of Hamilton's equation for the momentum map
and the non-degeneracy 
of the symplectic form $\omega$. 
Thus, (ii) is equivalent to~(iii).

Consider the isomorphism $\g^* \xrightarrow{\simeq} \g$ 
that is induced by the inner product.
Let $\widehat{\beta}$ denote the image of $\beta$.
Then $T_\beta$ is the closure of the one parameter subgroup
generated by $\widehat{\beta}$,
and so (iv) is equivalent to the condition 
that $\widehat{\beta}$ belong to the infinitesimal stabilizer at $p$.
Applying the isomorphism $\g \to \g^*$,
the relation $\widehat{\beta} \in \h$ becomes (iii).
\end{proof}

\bigskip

\begin{Example}
We consider the linear action $T^2\acts \C^3$ with weights $(1,0)$, $(0,1)$ 
and $(1,-1)$:
$$
(a,b)\cdot (z_1,z_2,z_3) = (az_1, bz_2, ab^{-1}z_3).
$$
The quadratic momentum map for this action is
\begin{eqnarray*}
 Q \colon\thinspace \C^3 & \to & \R^2\\
(z_1,z_2,z_3) & \mapsto & \left( \frac{|z_1|^2}{2}+\frac{|z_3|^2}{2} \, , \, 
\frac{|z_2|^2}{2}-\frac{|z_3|^2}{2} \right).
\end{eqnarray*}
We shift it by $(-3, 1)$, to obtain the momentum map
$$
   \Phi \left( (z_1,z_2,z_3)\right) = 
        \left(-3+ \frac{|z_1|^2}{2}+\frac{|z_3|^2}{2} \, , \, 
               1+\frac{|z_2|^2}{2}-\frac{|z_3|^2}{2} \right).
$$
The momentum image is shown  in Figure~\ref{fig:C3} below.

\begin{figure}[ht]
%
%
\centering
\includegraphics[width=160pt]{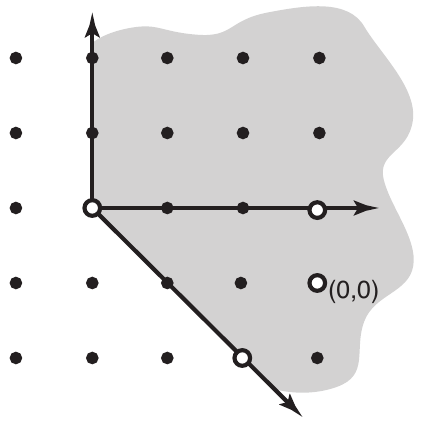}
\caption{The shaded region is the momentum map image, $\Phi\left(\C^3\right)$.
The lines are critical values for $\Phi$, and the large dots are the critical 
values for $||\Phi||^2$.}
\labell{fig:C3}
\end{figure}

A point $(z_1,z_2,z_3)\in \C^3$ is a critical point for $||\Phi||^2$ if and only
if it satisfies one of the following conditions:
\begin{enumerate}
\item[(i)] \ $z_1=z_2=z_3=0$;
\item[(ii)] \ $z_1=z_2=0$ \ and \ $\displaystyle{\frac{|z_3|^2}{2}} =2$; 
\item[(iii)] \ $z_2=z_3=0$ \ and \ 
   $\displaystyle{\frac{|z_1|^2}{2}} = 3$; or  
\item[(iv)] \ 
 $-3+\displaystyle{\frac{|z_1|^2}{2}} +\displaystyle{\frac{|z_3|^2}{2}}=0$ 
\ and \ 
$1+\displaystyle{\frac{|z_2|^2}{2}} -\displaystyle{\frac{|z_3|^2}{2}}=0$.

\end{enumerate}

We note that Condition (i) describes a single point, 
and each of Conditions (ii) and (iii) describes a single 
one--dimensional $T^2$--orbit.
Condition (iii) does not define 
an entire $\Phi$--level set, 
but each of the other conditions does.
Condition~(iv) defines a principal $T^2$ bundle over 
(the reduced space, which is)
a two--sphere.
\hfill $\diamondsuit$ 
\end{Example}

\bigskip

Let $\Phi_T \colon\thinspace M \to \t^*$ denote the momentum map
for a maximal torus $T$ of $G$; thus, $\Phi_T$ is the composition
of $\Phi \colon\thinspace M \to \g^*$ with the natural projection $\g^* \to \t^*$. 
Using the inner product, we also view $\t^*$ as a subspace of $\g^*$.

\begin{Lemma}[Kirwan {\cite[Lemma~3.1]{kirwan:book}}]  \labell{PhiT square}
Suppose that $\Phi(p) \in \t^*$.  Then  $p \in \Crit \| \Phi \|^2$
if and only if $p \in \Crit \| \Phi_T \|^2$.
\end{Lemma} 

\begin{proof}
Let $\beta = \Phi(p)$, and let $T_\beta$ denote the closure in $G$
of the one parameter subgroup that is generated by the element of $\g$
that corresponds to $\beta$ by the inner product.  
The assumption that $\beta \in \t^*$ implies that $T_\beta$ 
is contained in $T$.
The lemma then follows from the equivalence of (i) and (iv)
in Lemma~\ref{criterion}, applied to $\Phi$ and to $\Phi_T$.
\end{proof}

In preparation for proving Theorem~\ref{norm square},
we will examine linear symplectic actions of compact Lie groups.
We start with torus actions.

Recall that, for a torus $T$ with Lie algebra $\t$ and dual space $\t^*$,
the characters (homomorphisms $T \to S^1$) are determined
by their differentials at the identity.
Having identified the Lie algebra of $S^1$ with $\R$,
these differentials form the weight lattice $\t^*_\Z$ in $\t^*$.

\begin{Lemma} \labell{linear torus}
Fix a linear symplectic action of a torus $T$ on a symplectic vector space $V$.
Fix a $T$ invariant compatible complex structure on $V$.
Consider the decomposition of $V$ into weight spaces,
$$ V = \bigoplus_{\mu} V_\mu .$$
That is, for every weight $\mu \in \t^*_\Z$, 
denoting the corresponding character $T \to S^1$ by $a \mapsto a^\mu$, we have
$$ V_\mu = \left\{ z \in V \ \left| \ 
        a \cdot z = a^\mu z \text{ for all } a \in T \right. \right\} .$$
\begin{enumerate}
\item
Let $\Phi_T \colon V \to\thinspace \t^*$
be a momentum map for the $T$ action on $V$.
Fix an inner product on $\t^*$.
Then the set $\Phi_T \left( \Crit \| \Phi_T \|^2 \right)$ is finite.
\item
Let $Q \colon\thinspace V \to \h^*$ be the quadratic momentum map
for a linear symplectic action of a compact Lie group $H$ 
that commutes with the $T$ action
and preserves the complex structure.  
Then, for any $z \in V$, writing
$$ z = \sum_{\mu} z_\mu \quad , \quad z_\mu \in V_\mu, $$
we have
$$ Q(z) = \sum_\mu Q(z_\mu).$$
\end{enumerate}
\end{Lemma}

\medskip

\begin{proof}[Proof of Part (1) of Lemma~\ref{linear torus}]
We can take the indexing set for the 
decomposition into weight spaces
to be
$\calW := \left\{ \mu \in \t_{\Z} \ | \ V_\mu \neq \{ 0 \} \right\}$.
Because $V$ is finite dimensional, this set is finite.

Let $\beta = \Phi_T(0)$.  
For $z \in V$, writing
$$ z = \sum_{\mu \in \calW} z_\mu \quad , \quad z_\mu \in V_\mu ,$$
we have
\begin{equation*}
   \Phi_T(z) = \beta + \sum_{\mu \in \calW} \frac{|z_\mu|^2}{2}\mu .
\end{equation*}

By the equivalence of (i) and (ii) in Lemma~\ref{criterion}, 
we have that $z \in \Crit \| \Phi_T \|^2$
if and only if $\Phi_T(z) \perp \image d\Phi_T|_z$.
By the above formula for the momentum map,
$\Phi_T(z)$ belongs to the affine space
$$\beta + \text{span} \{ \mu \, | \, z_\mu \neq 0 \},$$
and $\image d\Phi_T|_z$ is the corresponding linear space
$\text{span} \{ \mu \, | \, z_\mu \neq 0 \}$.
So the condition $\Phi_T(z) \perp \image d\Phi_T|_z$ holds
if and only if $\Phi_T(z)$ is the foot of the perpendicular
from the origin to the affine space. 

For every subset $I$ of $\calW$, let $\beta_I$ denote 
the foot of the perpendicular from the origin in $\t^*$
to the affine space $\beta + \text{span} \{ \mu \, | \, \mu \in I \}$.
Then $\Phi_T(\Crit \| \Phi_T \|^2)$ is contained
in the finite set $\{ \beta_I \, | \, I \subset \calW \}$.
This proves Part (1) of Lemma~\ref{linear torus}.
\end{proof}

\begin{proof}[Proof of Part (2) of Lemma~\ref{linear torus}]
Denote the $\h$-action on $V$ by $\xi \colon z \mapsto \xi \cdot z$
for $\xi \in \h$.  
Then the quadratic momentum map is 
\begin{equation} \labell{Q from B}
 Q(z) = \frac{1}{2} B(z,z) ,
\end{equation}
where $B$ is the (symmetric) $\h^*$ valued bilinear form
whose components are given by
\begin{equation} \labell{formula for B}
 B^\xi(u,v) = \omega ( \xi \cdot u , v) 
         \qquad \text{for all $\xi \in \h$}.
\end{equation}

The weight spaces $V_\mu$ are $H$-invariant;
this follows from the fact that the $H$ action commutes with the $T$ action
and preserves the complex structure,
and it implies that
\begin{equation} \labell{invt}
\text{ for any $\mu \in \calW$ and $\xi \in \h$,
if $z_\mu \in V_\mu$ then also $\xi \cdot z_\mu \in V_\mu$.}
\end{equation}

The spaces $V_\mu$ are symplectically orthogonal:
\begin{equation} \labell{direct}
\text{
for any $\mu_1 \neq \mu_2$, 
if $\zeta_1 \in V_{\mu_1}$ and $\zeta_2 \in V_{\mu_2}$,
then $\omega(\zeta_1,\zeta_2) = 0$.
}
\end{equation}

By~\eqref{formula for B}, \eqref{invt}, and~\eqref{direct}, 
\begin{equation} \labell{perp}
\text{whenever $\mu_1 \neq \mu_2$, we have $B(z_{\mu_1},z_{\mu_2}) = 0$.}
\end{equation}

Now,
\begin{align*}
 Q(z) &= \frac{1}{2} B(z,z) \quad \text{by~\eqref{Q from B}} \\
 &= \sum_{\mu_1,\mu_2} \frac{1}{2} B(z_{\mu_1},z_{\mu_2}) 
              \quad \text{ because $B$ is bilinear and $z=\sum z_\mu$ } \\
 &= \sum_{\mu} \frac{1}{2} B(z_\mu,z_\mu)
  + \sum_{\mu_1 \neq \mu_2} \frac{1}{2} B(z_{\mu_1},z_{\mu_2}) & \\
 &= \sum_\mu Q(z_\mu) \quad \text{ by~\eqref{Q from B} and~\eqref{perp}.}
\end{align*}
This proves Part (2) of Lemma~\ref{linear torus}.
\end{proof}

Next, we examine linear symplectic actions of possibly non-abelian compact Lie groups,
with attention to a neighbourhood of the origin. 

\begin{Lemma} \labell{linear nonabelian}
Let $\Phi \colon V \to \h^*$ be a momentum map 
for a linear symplectic action of a compact Lie group $H$ 
on a symplectic vector space $V$.
Fix an $\Ad$-invariant inner product on $\h$.
Then there exist 
\begin{itemize}
\item
an $H$ invariant linear subspace $N$ of $V$,
\item
an $H$ invariant closed connected subset $C$ of $N$ that contains the origin,
and 
\item
an open neighbourhood $U$ of $C$ in $V$, 
\end{itemize}
such that
\begin{enumerate}
\item
$N$ is a minimizing manifold for $\| \Phi \|^2$ along $C$, and
\item
the intersection of $U$ with the critical set $\Crit \| \Phi \|^2$ 
is exactly $C$.
\end{enumerate}
\end{Lemma}

\begin{proof}[Construction of $C$ and $N$ for Lemma~\ref{linear nonabelian}]
Let $\beta = \Phi(0)$.  Then 
\begin{equation} \labell{Phi vs Q}
 \Phi(\cdot) = \beta + Q(\cdot) , 
\end{equation}
where $Q \colon\thinspace V \to \h^*$ is the quadratic momentum map.
Because $\Phi$ is equivariant, $\beta \in \h^*$ is fixed
under the coadjoint action of $H$.

Let $T_\beta$ denote the closure in $H$
of the one parameter subgroup that is generated by the element of $\h$
that is identified with $\beta$ by the inner product.
For each weight $\mu \in (\t_\beta)^*_\Z$, let
$V_\mu = \{ \mu \in V \ | \ a \cdot z = a^\mu z 
       \text{ for all } a \in T_\beta \}$
with respect to an $H$-invariant compatible complex structure.
The definition of $T_\beta$ implies that
if $\left< \mu , \beta \right> = 0$
then all the vectors in $V_\mu$ are fixed by $T_\beta$
and so $\mu = 0$.
So we can write the weight space decomposition as
$$ V \ = \ V^{T_\beta} 
\oplus  
\bigoplus_{\substack{ \mu \ \text{s.t.} \\ \left< \mu,\beta \right> > 0 }}
       V_\mu 
\oplus
\bigoplus_{\substack{ \mu \ \text{s.t.} \\ \left< \mu,\beta \right> < 0 }}
       V_\mu. $$

We set
$$ C = Q^{-1}(0) \cap V^{T_\beta} $$
and
$$ N = V^{T_\beta} \oplus 
\bigoplus_{\substack{ \mu \ \text{s.t.} \\ \left< \mu,\beta \right> > 0 }}
       V_\mu. $$

Because $\beta$ is fixed under the coadjoint action of $H$,
the torus $T_\beta$ is contained in the centre of $H$,
so the weight spaces $V_\mu$ are $H$ invariant.
So $N$ is an $H$ invariant linear subspace of $V$,
and $C$ is an $H$ invariant closed (conical, hence) connected subset of $N$
that contains the origin.
\end{proof}

\begin{proof}[Proof of Part (1) of Lemma~\ref{linear nonabelian}]
We will now show that  
$N$ is a minimizing manifold for $\| \Phi \|^2$ along $C$.

For $z \in V$, writing
$$ z = \sum_{\mu} z_\mu \quad , \quad z_\mu \in V_\mu ,$$
we have
\begin{equation} \labell{formula for Q}
Q(z) = \sum_\mu \frac{|z_\mu|^2}{2} \mu .
\end{equation}

We also have
\begin{align}
\nonumber
 \| \Phi (z) \|^2 
 & = \Big\| \beta + Q(z) \Big\|^2  \nonumber \\
\nonumber
 &= \| \beta \|^2 + 2 \left< \beta , Q(z) \right> + \| Q(z) \|^2 \\
\labell{star}
 &= \| \beta \|^2 + \sum_\mu |z_\mu|^2 \left< \mu , \beta \right>
 + \| Q(z) \|^2 \qquad \text{ by~\eqref{formula for Q}.}
\end{align}

Now suppose that $z \in N$.  
Then 
$$ z = z_0 
     + \sum_{\substack{ \mu \ \text{s.t.} \\ \left< \mu,\beta \right> > 0 }}
       z_\mu. $$
From~\eqref{star}, we get that
$$ \| \Phi(z) \|^2 \geq \| \beta \|^2 ,$$
with equality if and only if 
\begin{equation} \labell{two conditions}
 |z_\mu|^2 \left< \mu,\beta \right> = 0 \text{ for all } \mu
\quad \text{and} \quad Q(z) = 0.
\end{equation}
The first of these two conditions holds if and only if $ z \in V^{T_\beta}$. 
Thus, the conditions~\eqref{two conditions} hold exactly if $z \in C$.

We have shown that $\| \Phi(\cdot) \|^2|_N \geq \| \beta \|^2$,
with equality exactly at the points of~$C$.
That is, $N$ satisfies the second of the two conditions
for being a minimizing manifold for $\| \Phi \|^2$ along $C$.

\medskip

Because $\| \Phi (\cdot) \|^2 |_N$ attains its minimum exactly 
at the points of $C$,
the Hessian $\Hess \| \Phi \|^2 |_{T_xN} $ is positive semidefinite
at every $x \in C$.
It remains to show that, for all $x \in C$,
the Hessian $\Hess \| \Phi \|^2$ is negative definite
on a subspace of $T_xV$ that is complementary to $T_xN$.
We can take this subspace to be the image of
$ \displaystyle{ 
\bigoplus_{\substack{ \mu \ s.t. \\ \left< \mu , \beta \right> < 0}} } V_\mu $
under the natural identification of $V$ with $T_xV$.
Thus, for $x \in C$, we need to show that the Hessian of the map
$$ \left( \zeta \mapsto \| \Phi(x+\zeta) \|^2 \right) 
 \colon 
 \bigoplus_{\substack{\mu \ s.t. \\ \left< \mu , \beta \right> <0 }} V_\mu 
 \ \to \ \R $$
is negative definite at $\zeta = 0$.

Consider
$$ z = x + \zeta \quad \text{ with $x \in C$ and $\zeta \in 
  \bigoplus\limits_{\substack{\mu \ s.t.\\ \left< \mu,\beta \right> <0}} 
   V_\mu$} . $$
Because $x \in C$, we have $x \in V^{T_\beta}$. So $z_0 = x$, and
\begin{equation} \labell{eq-cases}
 |z_\mu|^2 \left< \mu , \beta \right>
 = \begin{cases}
 \displaystyle{
 |\zeta_\mu|^2 \left< \mu,\beta \right>  }
                   & \text{ if } \left< \mu,\beta \right> < 0 \\
 \ 0 & \text{ otherwise. }
   \end{cases} 
\end{equation}
Part (2) of Lemma~\ref{linear torus}, applied to the actions of $T_\beta$ 
and of $H$ on $V$, gives 
$Q(x+\zeta) = Q(x)+Q(\zeta)$.
Because $x \in C$, we have $Q(x)=0$.  So
\begin{equation} \labell{same Q}
Q(z) = Q(\zeta).
\end{equation}

Substituting~\eqref{eq-cases} and~\eqref{same Q} into~\eqref{star}, we get
$$ \| \Phi(x+\zeta) \|^2 = \| \beta \|^2
   + \sum_{\substack{\mu \ s.t.\\ \left< \mu,\beta \right> <0 }}
     |\zeta_\mu|^2 \left< \mu , \beta \right>  \\ 
 + \| Q(\zeta) \|^2.$$
Because $\zeta \mapsto \| Q(\zeta) \|^2$ is homogeneous of degree four,
the Hessian of the map \ $\zeta \mapsto \| \Phi(x+\zeta) \|^2$ \
at \ $\zeta =0$ \ is the bilinear form that corresponds to the quadratic form
$$ \zeta \mapsto 
   \sum_{\substack{\mu \ s.t.\\ \left< \mu,\beta \right> <0 }} 
   |\zeta_\mu|^2 \left< \mu,\beta \right> .$$
It is negative definite, as required.
This proves that $N$ also satisfies the first
of the two conditions for being a minimizing submanifold for $\| \Phi \|^2$
along $C$.
This completes the proof of Part (1) of Lemma~\ref{linear nonabelian}.
\end{proof}

\begin{proof}[Proof of Part (2) of Lemma~\ref{linear nonabelian}]

Let $T$ be a maximal torus in $H$ that contains $T_\beta$.
Let 
$$\Phi_T \colon\thinspace V \to \t^*$$ 
denote the momentum map for $T$.

Consider $\alpha \in \t^* \subset \h^*$,
and suppose that $\alpha$ is in $\Phi(\Crit \| \Phi \|^2)$.
Take any $z \in \Crit \| \Phi \|^2 $ with $\Phi(z) = \alpha$.
Because $\Phi(z) \in \t^*$,
\begin{itemize}
\item
$z \in \Crit (\| \Phi \|^2)$ if and only if $z \in \Crit (\| \Phi_T \|^2)$,
\qquad by Lemma \ref{PhiT square};
\item
$\Phi(z) = \Phi_T(z)$.
\end{itemize}
Thus, $\alpha$ is also in $\Phi_T( \Crit \| \Phi_T \|^2 )$.

By Part (1) of Lemma~\ref{linear torus},
applied to the $T$ action on $V$,
the set of such $\alpha$ is finite.
Note that this set contains $\beta$.
So we can write
$$ \Phi ( \Crit \| \Phi \|^2 ) \cap \t^*
 =  \{ \beta , \alpha_1, \ldots , \alpha_m \},$$
where $\alpha_1,\ldots,\alpha_m$ are different from $\beta$. 
Because $\Crit \| \Phi \|^2 $ is $H$ invariant and the momentum map
is $H$ equivariant, we conclude that
$$ \Phi ( \Crit \| \Phi \|^2 )
 =  \{ \beta \} \cup \bigcup_{j=1}^m \Ad^*(H)( \alpha_j ) .$$

The coadjoint orbits $\Ad^*(H)(\alpha_j)$
are closed in $\h^*$ and do not contain $\beta$,
so a sufficiently small neighbourhood of $\beta$ in $\g^*$
does not meet any of these orbits.
Let $U$ be the preimage in $V$ of such an open neighborhood of $\beta$.  Then 
\begin{align*}
   U \cap \Crit \| \Phi \|^2 & = \Phi^{-1}(\beta) \cap \Crit \| \Phi \|^2 \\
  & = Q^{-1}(0) \cap \Crit \| \Phi \|^2 \qquad \text{by~\eqref{Phi vs Q}} \\
  & = Q^{-1}(0) \cap V^{T_\beta} \qquad 
   \text{by the equivalence of (i) and (iv) in Lemma~\ref{criterion}} \\
  & = C ,
\end{align*}
as required.

This completes the proof of Part (2) of Lemma~\ref{linear nonabelian}.
\end{proof}

Next, we consider Hamiltonian $G$ models such as those that occur
in the local normal form theorem.
We will need to keep track only of the $G$ action and momentum map
on these models.

\begin{Definition} \labell{Hamiltonian G model}
Fix a compact Lie group $G$, an $\Ad(G)$ invariant inner product
on its Lie algebra $\g$, and the corresponding $\Ad^*(G)$ invariant
inner product on the dual space $\g^*$.
A \textbf{Hamiltonian} $\mathbf{G}$ \textbf{model} is a manifold $Y$, 
equipped with a map
$$ \Phi_Y \colon\thinspace Y \to \g^* ,$$
that are obtained from the following construction.

Fix an element $\beta \in \g^*$.
Let $G_\beta$ be the stabilizer of $\beta$
under the coadjoint action of $G$ on $\g^*$
and $\g_\beta$ the Lie algebra of $G_\beta$.
Let $H$ be a closed subgroup of $G_\beta$.
Let $V$ be a symplectic vector space on which $H$ acts linearly 
and symplectically with a quadratic momentum map 
$$ Q \colon\thinspace V \to \h^*.$$ 
Set
\begin{equation} \labell{G model}
   Y \ = \ G \times_H (V \times (\g_\beta/\h)^*). 
\end{equation}
Here, $H$ acts on $V$ through the given action
and on $(\g_\beta/\h)^*$ through the coadjoint representation,
and $Y$ is the quotient of $G \times V \times (\g_\beta/\h)^* $
by the $H$ action
$$ H \ni a \colon\thinspace 
   (g,z,\nu) \mapsto (g a^{-1} , a \cdot z , a \cdot \nu). $$
For $[ g,z,\nu ]$ in the model $G \times_H ( V \times (\g_\beta/\h)^* )$,
set
$$ \Phi_Y ([g,z,\nu]) \ = \ \Ad^*(g) \left( \beta + Q(z) + \nu \right) ,$$
where $\h^*$ and $(\g_\beta/\h)^*$ are identified with subspaces of $\g^*$
through the given $\Ad^*(G)$--invariant inner product. 

A Hamiltonian $G$ model $Y = G \times_H (V \times (\g_\beta/\h)^* )$ 
is \textbf{centred} if it satisfies one, hence both, of the following
equivalent conditions.
\begin{itemize}
\item
The value $\beta$ is fixed by $\Ad^*(G)$.
\item
The torus $T_\beta$ is contained in the centre of $G$.

\end{itemize}
\end{Definition}

\begin{Remark} \labell{rk:twoform} 
On any Hamiltonian $G$ model $Y=G \times_H (V \times (\g_\beta/\h)^*)$
there exists a $G$-invariant closed two-form $\omega_Y$ 
that is non-degenerate at the basepoint $[1,0,0]$
and for which the map $\Phi_Y \colon Y \to \g^*$ is a momentum map.
If the model is centred, there exists such a $\omega_Y$
that is everywhere non-degenerate,
and the central orbit $G \cdot [1,0,0]$ in $Y$ is isotropic with respect 
to $\omega_Y$.
\end{Remark}

We now examine centred Hamiltonian $G$ models.

\begin{Lemma}\labell{case of centred model}
Fix a centred Hamiltonian $G$ model $Y$ with momentum map
$$ \Phi_Y \colon\thinspace Y \to \g^* .$$
Assume that the basepoint $[1,0,0]$ of $Y$ is a critical point for 
$\|\Phi_Y\|^2$.
Then there exist 
\begin{itemize}
\item
a $G$ invariant closed connected subset $C$ of $Y$ 
that contains the basepoint $[1,0,0]$,
\item
a $G$ invariant closed connected submanifold $N$ of $Y$ that contains $C$,
and 
\item
a $G$ invariant open neighbourhood $U$ of $C$ in $Y$, 
\end{itemize}
such that
\begin{enumerate}
\item
$N$ is a minimizing manifold for $\| \Phi_Y \|^2$ along $C$, and
\item
the intersection of $U$ with the critical set $\Crit \| \Phi_Y \|^2$
is exactly $C$.
\end{enumerate}
\end{Lemma}

\begin{proof}
By the equivalence of (i) and (iii) in Lemma~\ref{criterion}, 
applied to the basepoint $[1,0,0]$ in the model $Y$,
the condition that $[1,0,0]$ is critical for $\| \Phi_Y \|^2$
is equivalent to
$$ \beta \in \h^*.$$
(Recall that we have identified $\h^*$ with a subspace of $\g^*$.)

Because $\beta$ is fixed by $\Ad^*(H)$
(as an element of $\g^*$ and hence as an element of $\h^*$),
$$ \beta + Q(\cdot)$$
is an (equivariant) momentum map for the $H$ action on $V$.
Applying Lemma~\ref{linear nonabelian} to it,
let $N_H$ be an $H$ invariant subspace of $V$,
let $C_H$ be a closed connected subset of $N_H$
that contains the origin such that 
$N_H$ is a minimizing manifold for $\| \beta + Q(\cdot) \|^2$ along $C_H$,
and let $U_H$ be an $H$ invariant open neighbourhood of $C_H$ in $V$
whose intersection with $\Crit \| \beta + Q(\cdot) \|^2$ is exactly $C_H$.

Set 
$$ N = G \times_H \left( N_H \times (\g_\beta/\h)^* \right),$$
$$ C = G \times_H \left( C_H \times \{ 0 \} \right), $$
and
$$ U = G \times_H \left( U_H \times (\g_\beta/\h)^* \right). $$

We have
\begin{equation} \labell{sum}
  \| \Phi_Y([g,z,\nu]) \|^2 = \| \beta + Q(z) + \nu \|^2
 = \| \beta + Q(z) \|^2 + \| \nu \|^2 .
\end{equation}
The first equality is by the formula for $\Phi_Y$ 
and since the norm on $\g^*$ is $\Ad^*(G)$ invariant.
The second equality is because $\beta$ and $Q(z)$ are in~$\h^*$,
and $\nu$ is in $(\g_\beta/\h)^*$, 
which, as a subspace of $\g^*$, is orthogonal to $\h^*$.
From~\eqref{sum} we deduce that 
\begin{equation} \labell{criterion in model}
 [g,z,\nu] \in \Crit \| \Phi_Y\|^2
\quad \text{ if and only if } \quad
 \nu = 0 \text{ and } z \in \Crit \| \beta + Q(\cdot) \|^2.
\end{equation}

The properties of $C_H$ and $U_H$, and~\eqref{criterion in model},
imply that $C$ is closed in the model $Y$
and $U$ is a neighbourhood of $C$ in the model $Y$
whose intersection with $\Crit \| \Phi_Y \|^2$ is equal to $C$.

Because $\| \beta + Q(\cdot) \|^2 |_{N_H}$
attains its minimum exactly on $C_H$,
and by~\eqref{sum},
we conclude that $\| \Phi_Y \|^2|_{N}$ attains its minimum exactly 
on $C$.  So $N$ satisfies the second of the two conditions
for being a minimizing manifold for $\| \Phi_Y \|^2$ along $C$.

As in the proof of Lemma~\ref{linear nonabelian}, let $E_H$ be 
an $H$ invariant complementary subspace to $N_H$ in $V$, such that,
for any $x_H \in C_H$, 
the Hessian at $\zeta = 0$ of the function
\begin{equation} \labell{on slice}
 ( \zeta \mapsto \| \beta + Q(x_H+\zeta)\|^2 ) \colon \ E_H \to \R 
\end{equation}
is negative definite.
Now take an arbitrary point of $C$;
write it as $x = [g,x_H,0]$ with $g \in G$ and $x_H \in C_H$.
The map
$\zeta \mapsto [g,x_H+\zeta,0]$, from a sufficiently small neighbourhood
of the origin in $E_H$ to $Y$,
provides a transverse slice to $N$ at $x$.
The pullback of $\| \Phi_Y \|^2$ by this map 
is the map~\eqref{on slice},
whose Hessian at $\zeta = 0$ is negative definite.  
So $N$ also satisfies the first of the two conditions
for being a minimizing manifold for $\| \Phi_Y \|^2$ along $C$.
\end{proof}

We now examine Hamiltonian $G$ models that are not necessarily centred.

\begin{Lemma}\labell{case of model}
Fix a Hamiltonian $G$ model $Y$ with momentum map
$$ \Phi_Y \colon\thinspace Y \to \g^* .$$
Assume that the basepoint $[1,0,0]$ of $Y$ is a critical point for $\Phi_Y$.
Then there exist 
\begin{itemize}
\item
a $G$ invariant closed connected subset $C$ of $Y$ 
that contains the basepoint $[1,0,0]$,
\item
a $G$ invariant closed connected submanifold $N$ of $Y$ that contains $C$,
and
\item
a $G$ invariant open neighbourhood $U$ of the basepoint $[1,0,0]$ in $Y$, 
\end{itemize}
such that
\begin{enumerate}
\item
$N$ is a minimizing manifold for $\| \Phi_Y \|^2$ along $C$, and
\item
the intersection of $U$ with the critical set $\Crit \| \Phi_Y \|^2$
is exactly $U \cap C$.
\end{enumerate}
\end{Lemma}

\begin{proof}
We use the notation of Definition~\ref{Hamiltonian G model}.
We have the centred Hamiltonian $G_\beta$ model
$$ Y' = G_\beta \times_H (V \times (\g_\beta/\h)^*) $$
with the momentum map
$$ \Phi' \colon\thinspace Y' \to \g_\beta^* \quad , \quad
  \Phi'([g,z,\nu]) = \Ad^*(g)(\beta+\Phi_V(z)+\nu)  .$$

We can identify
\begin{equation} \labell{identify}
 Y = G \times_{G_\beta} Y' ,
\end{equation}
exhibiting $Y$ as bundle with fibre $Y'$
and base $G/G_\beta$. (The base is also naturally identified 
with the coadjoint orbit through $\beta$.)

Let $U_{\g_\beta^*}$ be a $G_\beta$ invariant neighbourhood of $\beta$ 
in $\g_\beta^*$,
(where we have identified $\g_\beta^*$ with a subspace of $\g^*$,)
that is sufficiently small so that the ``sweeping map"
\begin{equation} \labell{the map}
 \left( [g,\eta] \mapsto \Ad^*(g)(\eta)  \right)
 \colon\thinspace G \times_{G_\beta} U_{\g_\beta^*} \to \g^*
\end{equation}
is a $G$ equivariant open embedding.

The restriction of
$$ \Phi_Y \colon Y \to \g^* $$
to the open subset 
$$ G \times_{G_\beta} \left( {\Phi'}^{-1} (U_{\g_\beta^*}) \right) $$
(with the identification~\eqref{identify})
is given by the composition of the map
$$ \left( [g,y'] \mapsto [g,\Phi'(y')] \right) \colon\thinspace
 G \times_{G_\beta} Y' \to G \times_{G_\beta} \g_\beta^* $$
with the open embedding~\eqref{the map}.  
So, for any $y' \in {\Phi'}^{-1} (U_{\g_\beta^*}) $, we have
$$ [g,y'] \in \Crit \| \Phi_Y \|^2 \quad \text{ if and only if } \quad
  y' \in \Crit \| \Phi' \|^2. $$

In particular, because the basepoint $[1,0,0]$ of $Y$
is critical for $\| \Phi_Y \|^2$,
the basepoint $[1,0,0]$ of $Y'$ is also critical for $\| \Phi' \|^2$.
(Alternatively, this fact follows from the equivalence
of (i) and (iii) in Lemma~\ref{criterion},
applied to the basepoint $[1,0,0]$ in the model $Y'$
and to the basepoint $[1,0,0]$ in the model $Y$.
In both cases, being critical for the norm squared of the momentum map
is equivalent to $\beta \in \h^*$,
where we embed $\h^* \subset \g_\beta^* \subset \g^*$.)

Applying Lemma~\ref{case of centred model} to the centred
Hamiltonian $G_\beta$ model $Y'$,
let $C'$ be a $G_\beta$ invariant closed connected subset of $Y'$
that contains the basepoint $[1,0,0]$,
let $N'$ be a $G_\beta$ invariant closed connected submanifold of $Y'$
that contains $C'$, and 
let $U'$ be a $G_\beta$ invariant open neighbourhood of $C'$ in $Y'$,
such that
\begin{enumerate}
\item
$N'$ is a minimizing manifold for $\|\Phi'\|^2$ along $C'$, and
\item
the intersection of $U'$ with the critical set $\Crit \| \Phi' \|^2$
is exactly $C'$.
\end{enumerate}

With the identification~\eqref{identify}, set
$$ C = G \times_{G_\beta} C' $$
and 
$$ N = G \times_{G_\beta} N' .$$
Then $C$ is a $G$ invariant closed connected subset of $Y$
that contains the basepoint $[1,0,0]$,
and $N$ is a $G$ invariant closed connected submanifold of $Y$
that contains~$C$.
Because $\| \Phi_Y ([g,y']) \|^2 = \| \Phi'(y') \|^2$,
the fact that $N'$ is a minimizing manifold for $\| \Phi' \|^2$ along $C'$
implies that $N$ is a minimizing manifold for $\| \Phi_Y \|^2$ along $C$.

Set 
$$ U = G \times_{G_\beta} \left( 
		       U' \cap {\Phi'}^{-1} (U_{\g_\beta^*}) \right) .$$
As noted earlier,
for $y=[g,y'] \in U$, we have
$$ y \in \Crit \| \Phi_Y \|^2 \quad \text{ if and only if } \quad
  y' \in \Crit \| \Phi' \|^2.$$

The fact that $U' \cap \Crit \| \Phi' \|^2 = C'$
then implies that $U \cap \Crit \| \Phi_Y \|^2 = U \cap C$.

\end{proof}

We can now prove Theorem~\ref{norm square}.

\begin{proof}[Proof of Theorem~\ref{norm square}]
Recall that
$\Phi \colon\thinspace M \to \g^*$ is an (equivariant) momentum map
for the action of a compact Lie group $G$ on a symplectic manifold $M$
and that we have fixed an $\Ad$-invariant inner product on $\g$
and the induced inner product on~$\g^*$.

Fix any critical point $p \in \Crit \| \Phi \|^2$.

By the local normal form theorem (see Guillemin-Sternberg \cite{GS1984} 
or Marle \cite{marle1985}; also see Sjamaar \cite[p.~77--78]{sjamaar}),
for each $G$ orbit $G \cdot p$ in $M$ there exists a 
Hamiltonian $G$ model 
\begin{equation} \labell{G model}
   Y = G \times_H (V \times (\g_\beta/\h)^*) \quad , \quad
   \Phi_Y \colon\thinspace Y \to \g^* , 
\end{equation}
and a $G$ equivariant diffeomorphism 
\begin{equation} \labell{lnf}
 \calO_M \to \calO_Y 
\end{equation}
from a neighbourhood $\calO_M$ of the orbit $G \cdot p$ in $M$
to a neighbourhood $\calO_Y$ of the central orbit $G \cdot [1,0,0]$ in $Y$
that takes $p$ to $[1,0,0]$ and whose composition with $\Phi_Y$ is $\Phi$.

Because $p$ is a critical point for $\| \Phi \|^2$,
the basepoint $[1,0,0]$ of $Y$ is a critical point for $\| \Phi_Y \|^2$.
By Lemma~\ref{case of model}, 
there exist a $G$-invariant closed connected subset $C_Y$ of $Y$,
a $G$-invariant closed connected submanifold $N_Y$ of $Y$ that contains $C_Y$,
and a $G$-invariant open neighbourhood $U_Y$ of the basepoint $[1,0,0]$ in $Y$,
such that
$N_Y$ is a minimizing manifold for $\| \Phi_Y \|^2$ along $C_Y$,
and such that
the intersection of $U_Y$ with the critical set $\Crit \| \Phi_Y \|^2$
is exactly $U_Y \cap C_Y$.

Let $C$, $N$, and $U$ be the preimages 
of $\calO_Y \cap C_Y$, 
of $\calO_Y \cap N_Y$,
and of $\calO_Y \cap U_Y$, 
under the local normal form diffeomorphism~\eqref{lnf}.
Then $N$ is a minimizing manifold for $\Crit \| \Phi \|^2$ along $C$,
and $U$ is an open neighbourhood of $G \cdot p$
whose intersection with $\Crit \| \Phi \|^2$ is $U \cap C$.

Because the critical point $p$ was arbitrary,
this shows that every point in $M$
has a neighbourhood on which $\|  \Phi \|^2 $ is minimally degenerate.
By Theorem~\ref{main}, 
we conclude that $\|\Phi\|^2$ is minimally degenerate.
This completes the proof of Theorem~\ref{norm square}.
\end{proof}

\section{Morse theoretic consequences}
\labell{sec:Morse}

For the convenience of the reader, and to put our work in context, 
we now recall the main topological consequences of the fact that 
the norm-square of the momentum map is minimally degenerate.

Let $M$ be a compact manifold 
and $f \colon\thinspace M \to \R$ a smooth function 
that is minimally degenerate.
So the critical set $\Crit f$ is a locally finite union
of closed subsets~$C$, on each of which $f$ is constant,
and, for each such critical set $C$, there exists a minimizing submanifold
$N_C$ for $f$ along $C$.

Kirwan developed the analytic tools necessary to extend results about
Morse functions to minimally degenerate functions \cite[\S 10]{kirwan:book}.
There exists a Riemannian metric on $M$ for which
the gradient vector field of $f$ is tangent to the minimizing manifold $N_C$
on a neighbourhood of $C$, for each critical set $C$.
For such a Riemannian metric, we let 
$$
S_C := \left\{ x\in M \ \left| \ \begin{array}{c} 
\mbox{the gradient trajectory for $-f$ starting at $x$ } \\
\mbox{ has a limit point in } C\end{array}  \right.  \right\}.
$$
Kirwan then established the following facts.
First, $S_C$ is a submanifold of $M$ that coincides with $N_C$ near $C$. 
Moreover, the inclusion map $C \subset S_C$ induces an 
isomorphism in \v Cech cohomology \cite[Lemma 10.17]{kirwan:book}.
Finally, the submanifolds $S_C$ give a decomposition of $M$ 
into a disjoint union
\begin{equation}\labell{eq:decomp}
M = \bigsqcup_C S_C
\end{equation}
that satisfies the  frontier condition
$$ \closure(S_C) \ {\mbox{\Large{$\subset$}}} \ 
S_C\ \  \cup 
\bigcup_{\substack{C' \text{ s.t. } \\ f(C') > f(C)}} S_{C'}.$$

In the presence of a compact connected group action that preserves~$f$, 
we can choose the Riemannian metric to be invariant,
and then the submanifolds $S_C$ are invariant
and the inclusions $C \to S_C$ induce isomorphisms in equivariant cohomology.
The decomposition \eqref{eq:decomp} gives rise to the Morse inequalities,
and, in the presence of a group action, the equivariant Morse inequalities.
When $f=||\Phi||^2$ is the norm-square of the momentum map for a Hamiltonian
action of a compact Lie group, \eqref{eq:decomp}
leads to Kirwan surjectivity  \cite[pp.\ 31--34]{kirwan:book}.

\begin{Theorem}[Kirwan surjectivity] \labell{thm:surj}
Let a compact Lie group $G$ on a compact symplectic manifold $M$
with momentum map $\Phi \colon\thinspace M \to \g^*$.
Then the inclusion $\Phi^{-1}(0)\to M$ induces
a surjection in equivariant cohomology
$$
H_G^*(M;\Q)\to H_G^*(\Phi^{-1}(0);\Q).
$$
\end{Theorem}

For a sufficiently small ball $U$ about the origin in $\g^*$,
there exists an equivariant deformation retraction from the preimage
$\Phi^{-1}(U)$ to the level set $\Phi^{-1}(0)$.  If $0$ is a
regular value of $\Phi$, this follows from the tubular neighbourhood
theorem; in general, it follows from the results of \cite{lerman:2005}.
Theorem~\ref{thm:surj} then follows 
from the proof of  \cite[Lemma~2.18]{kirwan:book}. 
We note that 
a key technical tool in the proof  of  \cite[Lemma~2.18]{kirwan:book}
is the Atiyah--Bott Lemma \cite[Proposition 13.4]{AB}, which
provides a condition that guarantees 
an equivariant Euler class to be a non-zero divisor.
The Atiyah--Bott Lemma may be applied 
to the (normal bundles of the) strata $S_C$, 
but not necessarily to the critical sets themselves.

\begin{Remark}
Lerman's paper~\cite{lerman:2005} gives a retraction 
from $\Phi^{-1}(U)$ to $\Phi^{-1}(0)$ which is
an equivariant homotopy inverse to the inclusion map
of $\Phi^{-1}(0)$ in $\Phi^{-1}(U)$.
Usually this retraction is only continuous and not smooth.
We believe that the inclusion map 
of $\Phi^{-1}(0)$ in $\Phi^{-1}(U)$
does have a smooth equivariant homotopy inverse
(whose restriction to $\Phi^{-1}(0)$ is homotopic
to the identity but not equal to the identity).
Details will appear elsewhere.
\end{Remark}

\end{document}